\documentclass[11pt,reqno]{amsart}
\usepackage{enumerate}
\usepackage{amsfonts}
\usepackage{amsmath}
\usepackage{amsthm}
\usepackage{amssymb}
\usepackage{latexsym}
\usepackage{multicol}
\usepackage{verbatim}
\usepackage{amscd}

\advance\textwidth by 1.2in \advance\oddsidemargin by -.6in \advance\evensidemargin by -.6in
\newtheorem*{cor}{Corollary}
\newtheorem*{lem}{Lemma}
\newtheorem*{prop}{Proposition}

\theoremstyle{definition}

\theoremstyle{definition}
\newtheorem{thm}{Theorem}

\newenvironment{pf}{\proof}{\endproof}
\newcounter{cnt}
\newenvironment{enumerit}{\begin{list}{{\hfill\rm(\roman{cnt})\hfill}}{%
\settowidth{\labelwidth}{{\rm(iv)}}\leftmargin=\labelwidth%
\advance\leftmargin by \labelsep\rightmargin=0pt\usecounter{cnt}}}{\end{list}}

\theoremstyle{remark}

\numberwithin{equation}{section} 

\begin{document}

\newcommand{\thmref}[1]{Theorem~\ref{#1}}
\newcommand{\secref}[1]{Section~\ref{#1}}
\newcommand{\lemref}[1]{Lemma~\ref{#1}}
\newcommand{\propref}[1]{Proposition~\ref{#1}}
\newcommand{\corref}[1]{Corollary~\ref{#1}}
\newcommand{\remref}[1]{Remark~\ref{#1}}
\newcommand{\defref}[1]{Definition~\ref{#1}}
\newcommand{\er}[1]{(\ref{#1})}
\newcommand{\id}{\operatorname{id}}
\newcommand{\tensor}{\otimes}
\newcommand{\nc}{\newcommand}
\newcommand{\rnc}{\renewcommand}
\newcommand{\qbinom}[2]{\genfrac[]{0pt}0{#1}{#2}}
\nc{\cal}{\mathcal} \nc{\goth}{\mathfrak} \rnc{\bold}{\mathbf}
\renewcommand{\frak}{\mathfrak}
\newcommand{\supp}{\operatorname{supp}}
\newcommand{\desc}{\operatorname{desc}}
\newcommand{\Maj}{\operatorname{Maj}}
\renewcommand{\Bbb}{\mathbb}
\nc\bomega{{\mbox{\boldmath $\omega$}}} \nc\bpsi{{\mbox{\boldmath $\Psi$}}}
 \nc\balpha{{\mbox{\boldmath $\alpha$}}}
 \nc\bpi{{\mbox{\boldmath $\pi$}}}

\newcommand{\lie}[1]{\mathfrak{#1}}
\makeatletter
\def\section{\def\@secnumfont{\mdseries}\@startsection{section}{1}%
  \z@{.7\linespacing\@plus\linespacing}{.5\linespacing}%
  {\normalfont\scshape\centering}}
\def\subsection{\def\@secnumfont{\bfseries}\@startsection{subsection}{2}%
  {\parindent}{.5\linespacing\@plus.7\linespacing}{-.5em}%
  {\normalfont\bfseries}}
\makeatother
\def\subl#1{\subsection{}\label{#1}}
 \nc{\Hom}{\operatorname{Hom}}
\nc{\End}{\operatorname{End}} \nc{\wh}[1]{\widehat{#1}} \nc{\Ext}{\operatorname{Ext}} \nc{\ch}{\text{ch}} \nc{\ev}{\text{ev}}
\nc{\Ob}{\operatorname{Ob}} \nc{\soc}{\operatorname{soc}} \nc{\rad}{\operatorname{rad}}
\def\Im{\operatorname{Im}}
\def\gr{\operatorname{gr}}
\def\wt{\operatorname{wt}}

\nc\eva{\text{ev}_a} \nc{\krsm}{KR^\sigma(m\omega_i)} \nc{\krsmzero}{KR^\sigma(m_0\omega_i)} \nc{\krsmone}{KR^\sigma(m_1\omega_i)}
 \nc{\vsim}{v^\sigma_{i,m}}

 \nc{\Cal}{\cal} \nc{\Xp}[1]{X^+(#1)} \nc{\Xm}[1]{X^-(#1)}
\nc{\on}{\operatorname} \nc{\Z}{{\bold Z}} \nc{\J}{{\cal J}} \nc{\C}{{\bold C}} \nc{\Q}{{\bold Q}}
\renewcommand{\P}{{\cal P}}
\nc{\N}{{\Bbb N}} \nc\boa{\bold a} \nc\bob{\bold b} \nc\boc{\bold c} \nc\bod{\bold d} \nc\boe{\bold e} \nc\bof{\bold f} \nc\bog{\bold g}
\nc\boh{\bold h} \nc\boi{\bold i} \nc\boj{\bold j} \nc\bok{\bold k} \nc\bol{\bold l} \nc\bom{\bold m} \nc\bon{\bold n} \nc\boo{\bold o}
\nc\bop{\bold p} \nc\boq{\bold q} \nc\bor{\bold r} \nc\bos{\bold s} \nc\bou{\bold u} \nc\bov{\bold v} \nc\bow{\bold w} \nc\boz{\bold z}
\nc\boy{\bold y} \nc\ba{\bold A} \nc\bb{\bold B} \nc\bc{\bold C} \nc\bd{\bold D} \nc\be{\bold E} \nc\bg{\bold G} \nc\bh{\bold H} \nc\bi{\bold I}
\nc\bj{\bold J} \nc\bk{\bold K} \nc\bl{\bold L} \nc\bm{\bold M} \nc\bn{\bold N} \nc\bo{\bold O} \nc\bp{\bold P} \nc\bq{\bold Q} \nc\br{\bold R}
\nc\bs{\bold S} \nc\bt{\bold T} \nc\bu{\bold U} \nc\bv{\bold V} \nc\bw{\bold W} \nc\bz{\bold Z} \nc\bx{\bold x}

\nc\beq{\begin{equation}} \nc\enq{\end{equation}} \nc\lan{\langle} \nc\ran{\rangle} \nc\bsl{\backslash} \nc\mto{\mapsto}
\nc\lra{\leftrightarrow} \nc\hra{\hookrightarrow} \nc\sm{\smallmatrix} \nc\esm{\endsmallmatrix} \nc\sub{\subset} \nc\ti{\tilde} \nc\nl{\newline}
\nc\fra{\frac} \nc\und{\underline} \nc\ov{\overline} \nc\ot{\otimes} \nc\bbq{\bar{\bq}_l} \nc\bcc{\thickfracwithdelims[]\thickness0}
\nc\ad{\text{\rm ad}} \nc\Ad{\text{\rm Ad}}  \nc\Ind{\text{\rm Ind}} \nc\Res{\text{\rm Res}} \nc\Ker{\text{\rm Ker}} \rnc\Im{\text{Im}}
\nc\sgn{\text{\rm sgn}} \nc\tr{\text{\rm tr}} \nc\Tr{\text{\rm Tr}}
 \nc\card{\text{\rm card}}
\nc\bst{{}^\bigstar\!} \nc\he{\heartsuit} \nc\clu{\clubsuit} \nc\spa{\spadesuit} \nc\di{\diamond}

\nc\al{\alpha} \nc\bet{\beta} \nc\ga{\gamma} \nc\de{\delta} \nc\ep{\epsilon} \nc\io{\iota} \nc\om{\omega} \nc\si{\ensuremath{\sigma}}
\nc\tildsi{\ensuremath{\widetilde{\sigma}}} \rnc\th{\theta} \nc\ka{\kappa} \nc\la{\lambda} \nc\ze{\zeta}

\nc\vp{\varpi} \nc\vt{\vartheta} \nc\vr{\varrho}

\nc\Ga{\Gamma} \nc\De{\Delta} \nc\Om{\Omega} \nc\Si{\Sigma} \nc\Th{\Theta} \nc\La{\Lambda}

\nc\blambda{{\mbox{\boldmath $\Lambda$}}}

\nc\e[1]{E_{#1}} \nc\ei[1]{E_{\delta - \alpha_{#1}}} \nc\esi[1]{E_{s \delta - \alpha_{#1}}} \nc\eri[1]{E_{r \delta - \alpha_{#1}}}
\nc\ed[2][]{E_{#1 \delta,#2}} \nc\ekd[1]{E_{k \delta,#1}} \nc\emd[1]{E_{m \delta,#1}} \nc\erd[1]{E_{r \delta,#1}}  \nc\ef[1]{F_{#1}}
\nc\efi[1]{F_{\delta - \alpha_{#1}}} \nc\efsi[1]{F_{s \delta - \alpha_{#1}}} \nc\efri[1]{F_{r \delta - \alpha_{#1}}} \nc\efd[2][]{F_{#1
\delta,#2}} \nc\efkd[1]{F_{k \delta,#1}} \nc\efmd[1]{F_{m \delta,#1}} \nc\efrd[1]{F_{r \delta,#1}} \nc{\ug}{\bu^{fin}}
\nc\frakmu{\ensuremath{\lambda(\frac{h_\mu}{2})}} \nc\frakhn{\ensuremath{\lambda(\frac{h_n}{2})}}
\nc\frakshort{\ensuremath{\lambda(\frac{h_{\theta_s}}{2})}} \nc\hatshort{\ensuremath{\lambda(\hat h_{\theta_s})}}
\nc\hathn{\ensuremath{\lambda(\hat h_{n})}} \nc\nonzerowts{(R_0)^+ \cup (R_1)^+} \nc\X{\mathbf X}

\nc\fa{\frak a} \nc\fb{\frak b} \nc\fc{\frak c} \nc\fd{\frak d} \nc\fe{\frak e} \nc\ff{\frak f} \nc\fg{\frak g} \nc\fh{\frak h} \nc\fj{\frak j}
\nc\fk{\frak k} \nc\fl{\frak l} \nc\fm{\frak m} \nc\fn{\frak n} \nc\fo{\frak o} \nc\fp{\frak p} \nc\fq{\frak q} \nc\fr{\frak r} \nc\fs{\frak s}
\nc\ft{\frak t} \nc\fu{\frak u} \nc\fv{\frak v} \nc\fz{\frak z} \nc\fx{\frak x} \nc\fy{\frak y}

\nc\fA{\frak A} \nc\fB{\frak B} \nc\fC{\frak C} \nc\fD{\frak D} \nc\fE{\frak E} \nc\fF{\frak F} \nc\fG{\frak G} \nc\fH{\frak H} \nc\fJ{\frak J}
\nc\fK{\frak K} \nc\fL{\frak L} \nc\fM{\frak M} \nc\fN{\frak N} \nc\fO{\frak O} \nc\fP{\frak P} \nc\fQ{\frak Q} \nc\fR{\frak R} \nc\fS{\frak S}
\nc\fT{\frak T} \nc\fU{\frak U} \nc\fV{\frak V} \nc\fZ{\frak Z} \nc\fX{\frak X} \nc\fY{\frak Y} \nc\tfi{\ti{\Phi}} \nc\bF{\bold F}
\nc\bbz{\mathbb{Z}}

\nc\ua{\bold U_\A}

\nc\qinti[1]{[#1]_i} \nc\q[1]{[#1]_q}
\nc\xpm[2]{E_{#2 \delta \pm \alpha_#1}}  
\nc\xmp[2]{E_{#2 \delta \mp \alpha_#1}} \nc\xp[2]{E_{#2 \delta + \alpha_{#1}}} \nc\xm[2]{E_{#2 \delta - \alpha_{#1}}} \nc\hik{\ed{k}{i}}
\nc\hjl{\ed{l}{j}} \nc\qcoeff[3]{\left[ \begin{smallmatrix} {#1}& \\ {#2}& \end{smallmatrix} \negthickspace \right]_{#3}} \nc\qi{q} \nc\qj{q}

\nc\ufdm{{_\ca\bu}_{\rm fd}^{\le 0}}


\nc\isom{\cong}

\nc{\pone}{{\Bbb C}{\Bbb P}^1} \nc{\pa}{\partial}
\def\h{\cal h}
\def\L{\cal L}
\nc{\F}{{\cal F}} \nc{\Sym}{{\goth S}} \nc{\A}{{\cal A}} \nc{\arr}{\rightarrow} \nc{\larr}{\longrightarrow}

\nc{\ri}{\rangle} \nc{\lef}{\langle} \nc{\W}{{\cal W}} \nc{\uqatwoatone}{{U_{q,1}}(\su)} \nc{\uqtwo}{U_q(\goth{sl}_2)} \nc{\dij}{\delta_{ij}}
\nc{\divei}{E_{\alpha_i}^{(n)}} \nc{\divfi}{F_{\alpha_i}^{(n)}} \nc{\Lzero}{\Lambda_0} \nc{\Lone}{\Lambda_1} \nc{\ve}{\varepsilon}
\nc{\phioneminusi}{\Phi^{(1-i,i)}} \nc{\phioneminusistar}{\Phi^{* (1-i,i)}} \nc{\phii}{\Phi^{(i,1-i)}} \nc{\Li}{\Lambda_i}
\nc{\Loneminusi}{\Lambda_{1-i}} \nc{\vtimesz}{v_\ve \otimes z^m}

\nc{\asltwo}{\widehat{\goth{sl}_2}} \nc\eh{\frak h^e} \nc\loopg{\ensuremath{L(\frak g)}} \nc\twistg{\ensuremath{L^{\sigma}(\frak g ) }}
\nc\twistatwo{\ensuremath{L^{\widetilde{\sigma}}(A_2)}} \nc\lambdatwo{\ensuremath{\frac{\lambda(h_0)}{2}}} \nc\eloopg{L^e(\frak g)}
\nc\ebu{\bu^e} \nc\lamhi{\ensuremath{\lambda(h_i)}} \nc\lamha{\ensuremath{\lambda(h_\alpha)}} \nc\lamhmu{\ensuremath{\lambda(h_\mu)}}

\nc\btildepi{\ensuremath{{\mbox{\boldmath$\tilde{\pi}$}}}} \nc\teb{\tilde E_\boc} \nc\tebp{\tilde E_{\boc'}}

\title{Weyl modules for the twisted loop algebras }
\author{Vyjayanthi Chari, Ghislain Fourier and Prasad Senesi}
\thanks{VC was partially supported by the NSF grant DMS-0500751}
\address{ VC:  Department of Mathematics, University of California,\
Riverside, CA 92521}
\address{GF: Mathematisches Institut, Universit\"at zu K\"oln, 50931 K\"oln}
\address{ PS:  Department of Mathematics, University of California,\
Riverside, CA 92521}

\begin{abstract}
  The notion of a Weyl module, previously defined for the untwisted affine
algebras, is extended here to the twisted affine algebras. We describe an
identification of the Weyl modules for the twisted affine algebras with
suitably chosen Weyl modules for the untwisted affine algebras. This
identification allows us to use known results in the untwisted case to compute
the dimensions and characters of the Weyl modules for the twisted algebras.
\end{abstract}

\pagestyle{plain} \maketitle
\section{Introduction} The notion of Weyl modules for the
untwisted affine Lie algebras was introduced in \cite{CPweyl} and
was motivated by an attempt to understand the category of finite
dimensional representations of the untwisted quantum affine
algebra. Specifically, the Weyl modules were conjectured to be the
$q=1$ limit of certain irreducible representations of the quantum
affine algebras. It was proved that the conjecture was true for
$\lie {sl_2}$ and that  this conjecture would follow if the
dimensions of the Weyl modules were known.  H. Nakajima has
pointed out recently that the dimension formula  follows  by using
results of \cite{BN} and \cite{K}.

Another approach to proving the dimension formula for  the Weyl
modules can be found in \cite{CL} for $\lie sl_n$ and in
\cite{FoL} for the general simply laced case. These   papers also
make the connection between Weyl modules and the Demazure modules
for affine Lie algebras and also with the fusion product defined
by \cite{FL}. The approach in these papers is rather simple and
show that one can study the  Weyl modules from a purely classical
viewpoint. Other points of interest and generalizations of these
can be found in \cite{FKL}.

We now turn our attention to the case of the twisted affine
algebras. None of the quantum machinery is available and in fact
there are rather few results on the category of finite dimensional
representations of the twisted quantum affine algebras
\cite{Akasaka}, \cite{CPtw}. These results do show however  that
one can make a similar conjecture; i.e that one can define a
notion of the Weyl module for the twisted affine Lie algebras such
that they are the specializations of irreducible modules in the
quantum case. To do this, one requires the Weyl modules to be
universal in a suitable sense. One of  the difficulties  is in the
case of the algebras of type $A_{2n}^{(2)}$, which are not built up
entirely of algebras isomorphic to $A_1^{(1)}$; and indeed one
needs to understand $A_2^{(2)}$ on its own. Thus, we use results
of \cite{FV}, \cite{Mitz} to arrive at the correct definition of
the Weyl modules.

The next question clearly is to determine the dimensions of the
Weyl modules and also their decomposition as modules for the
underlying finite--dimensional simple Lie algebra. In the
untwisted case these questions can be  answered either  by using
the fusion product of \cite{FL} or the fact that the modules are
specializations of modules for the quantum affine algebra. Both
these techniques are unavailable to us in the twisted case, as far as we know the
notion of fusion product does not admit a generalization to the
twisted algebras. We get around these difficulties by identifying
the Weyl modules for the twisted algebras $X_n^{(m)}$, $m>1$ with
suitably chosen Weyl modules for the untwisted algebra
$X_n^{(1)}$. We then use all the known results in the untwisted
case to complete our analysis of the twisted algebras. In
conclusion, we note that some of the methods we use in this paper
give simpler proofs of some of the results in \cite{CPweyl}.

\section{The untwisted loop algebras and the modules $W(\bpi)$.}
\subsection{} Throughout the paper $\bc$ (resp. $\bc^\times$) denotes the set of complex (resp. non--zero complex) numbers, and $\bz$ (resp. $\bz_+$)
 the set of integers (resp. non--negative) integers. Given a Lie algebra $\lie a$ we denote by $\bu(\lie a)$  the universal enveloping
 algebra of $\lie a$ and by $L(\lie a)$ denotes the loop algebra of $\lie a$. Specifically, we have
 \begin{equation*}L(\frak a) = \frak a\otimes \bc[t,t^{-1}],\end{equation*}
with commutator given by
\begin{equation*}
 [x\otimes t^r, y\otimes t^s]
=[x,y]\otimes t^{r+s}\end{equation*} for $x,y\in\frak a$, $r,s\in\bz$. We identify $\lie a$ with the subalgebra $\lie a\otimes 1$ of $L(\lie
a)$. Given $a\in\bc^\times$, we let $\tau_a:L(\lie a)\to L(\lie a)$ be the automorphism defined by extending $\tau_a(x\otimes t^k)=a^k(x\otimes
t^k)$ for all $x\in\lie g$, $k\in\bz$.

Given $\ell,N\in\bz_+$ and $\boa=(a_1,\cdots
,a_\ell)\in(\bc^\times)^\ell$ let $\lie a_{\boa,N}$  be the quotient
of $L(\lie a)$ by the ideal $\lie
a\otimes\prod_{k=1}^\ell(t-a_k)^N\bc[t,t^{-1}]$.
\begin{lem}\label{chinun} Let $\boa=(a_1,\cdots,a_\ell)\in(\bc^\times)^\ell$ be such that
$\boa$ has distinct coordinates. For all $N\in\bz_+$, we have
$$\lie a_{\boa,N}\cong\oplus_{r=1}^N\lie a_{a_r,N}.$$
\end{lem}
\begin{pf} Since $a_r\ne a_s$ if $1\le r\ne s\le \ell$, it is
standard that
$$\bc[t,t^{-1}]/\prod_{r=1}^\ell(t-a_r)^N\bc[t,t^{-1}]\cong\oplus_{r=1}^{\ell}\bc[t,t^{-1}]/(t-a_r)^N\bc[t,t^{-1}]$$  and the lemma now follows
trivially.
\end{pf}

\subsection{The simple Lie algebras and their representations} Let  $\frak{g}$ be any finite-dimensional  complex simple
Lie algebra and $\lie h$ a Cartan subalgebra of $\lie g$ and
$W_{\lie g}$ the corresponding Weyl group. Let  $R_{\lie g}$ be
the set of roots of $\frak g$ with respect to $\lie h$, $I_{\lie
g}$ an index set for  a set of simple roots (and hence also for
the fundamental weights), $R_{\lie g}^+$  the set of positive
roots, $Q_{\lie g}^+$ (resp. $P_{\lie g}^+$) the $\bz_+$ span of
the simple roots (resp. fundamental weights) and $\theta_{\lie g}$
be the highest root in $R^+_{\lie g}$. Given $\alpha\in R_{\lie
g}$ let $\lie g_\alpha$ be the corresponding root space, we have
$$\lie g=\lie n^-\oplus \lie h\oplus\lie n^+,\ \ \lie n^\pm=\bigoplus_{\alpha\in R^+}\lie
g_{\pm\alpha}.$$  Fix a Chevalley basis $x^\pm_\alpha, h_\alpha$,
$\alpha\in R^+$ for $\lie g$   and set
$$x^\pm_{\alpha_i}=x^\pm_i,\ \ h_{\alpha_i} =h_i,\ \ i\in I.$$ In
particular for $i\in I$, $$[x^+_i,x^-_i]=h_i,\ \ [h_i,x^\pm_i]=\pm 2x_i^\pm.$$

Given  a finite--dimensional representation of $\lie g$ on a
complex vector space $V$, we can write $$V=\oplus_{\mu\in\lie
h^*}V_\mu,\ \ V_\mu=\{v\in V: hv=\mu(h)v\ \ \forall \ h\in\lie
h\}.$$ Set $\wt(V)=\{\mu\in\lie h^*: V_\mu\ne 0\}$. It is
well--known that
$$ V_\mu \ne 0\ \ \implies \mu\in P\ \ {\text {and}}\ \
w\mu\in\wt(V) \ \ \forall \ \ w\in W,$$ and that $V$ is isomorphic
to a direct sum of irreducible representations. The set of
isomorphism classes of irreducible finite--dimensional $\lie
g$--modules is in bijective correspondence with $P^+$ and
 for any $\lambda\in P^+$ let $V_{\lie
g}(\lambda)$ be an element of the corresponding isomorphism class.
Then $V_{\lie g}(\lambda)$ is generated by an element $v_\lambda$
satisfying the relations:
\begin{equation} \lie n^+.v_\lambda=0,\ \ hv_\lambda=\lambda(h)v_\lambda,\ \ (x_i^-)^{\lambda(h_i)+1}v_\lambda=0.
\end{equation}

\subsection{Identities in $\bu(L(\lie g)$} For $i\in I$ it is easy to see that  the elements
$\{x_i^\pm\otimes t^k, h_i\otimes t^k: k\in\bz_+\}$ span a
subalgebra of $L(\lie g)$ which is isomorphic to $L(\lie{sl_2})$.
 We shall need the following formal power series in $u$
 with coefficients in  $\bu(L(\lie{\lie g}))$. For $i\in I$, set
\begin{gather*}
\bop_i^{\pm}(u) =  \text{exp} \left( - \sum_{k=1}^{\infty} \frac{h_i\otimes t^k}{k} u^{k}\right),\\\quad \ \mathbf x^\pm_{i}(u) =
\sum_{k=0}^{\infty}(x_i^\pm\otimes t^k)u^{k+1}, \quad \widetilde{\mathbf x_i}^\pm(u) = \sum_{k=-\infty}^{\infty}(x_i^\pm\otimes t^k) u^{k+1}\\
\end{gather*}
Given a power series $\bof$ in $u$ with coefficients in an algebra
$A$, let $(\bof)_m$ be the coefficient of $u^m$ ($m\in\bz$). The
following result was proved in \cite[Lemma 7.5]{Ga}, (see
\cite[Lemma 1.3]{CPweyl} for the formulation in this notation).
\begin{lem}\label{Garland}Let $r\in\bz_+$. \begin{gather*}
 (x_i^+\otimes t)^{(r)}_{} (x_i^-\otimes 1)^{ (r+1)} =(-1)^r \left(
\mathbf x^-_{i}(u)   \bop_i^+(u) \right)_{r+1} \mod \bu(L(\lie{g}))\widetilde{\mathbf x_i}^+(u).\end{gather*}\hfill\qedsymbol
\end{lem}

\subsection{The monoid $\cal P^+$} Let $\cal P^+$  be the monoid of $I$--tuples   of
polynomials $\bpi=(\pi_1,\cdots,\pi_n)$ in an indeterminate $u$
with constant term one, with multiplication being defined
component wise. For $i\in I$  and $a\in\bc^\times$,  set
\begin{equation}\label{calp} \bpi_{i,a}=((1-au)^{\delta_{ij}} :j\in I)\in \cal
P^+,
\end{equation}
and for $ \lambda\in P^+$, set $$\ \  \bpi_{\lambda,a} =
\prod_{i\in I}(\bpi_{i,a})^{\lambda(h_i)},\ \ \lambda\ne 0.$$
Clearly any $\bpi^+\in\cal P^+$ can be written uniquely as a
product
$$\bpi^+=\prod_{k=1}^\ell\bpi_{\lambda_i,a_i},$$ for some
$\lambda_1,\cdots,\lambda_\ell\in  P^+$ and distinct elements $a_1,\cdots, a_\ell\in\bc^\times$ and in this case we set
$\bpi^-=\prod_{k=1}^\ell\bpi_{\lambda_i,{} \ {} a_i^{-1}}$. Define a map $\cal P^+\to P^+$ by $\bpi\to\lambda_\bpi=\sum_{i\in
I}\deg(\pi_i)\omega_i.$

\subsection{The modules $W(\bpi)$, $V(\bpi)$.}
 Given $\bpi=(\pi_i)_{i\in I}\in\cal P^+$, let $W(\bpi)$ be the $L(\lie g)$--module  generated by an element $w_\bpi$ with relations:
\begin{gather*} L(\lie n^+)w_{\bpi}=0,\ \ hw_{\bpi}=\lambda_\bpi(h)
w_{\bpi},\ \ (x^-_i)^{\lambda_\bpi(h_i)+1}w_{\bpi}=0,\\
\left(\bop^\pm_{i}(u)-\pi^\pm_i(u)\right)w_{\bpi}=0,
\end{gather*}where $\lambda_\bpi=\sum_{i\in I}(\deg\pi_i)\omega_i$,  $\bpi^+=\bpi$, $i\in I$ and $h\in\lie h$.
It is not hard to see that if we write $\bpi=\prod_{k=1}^\ell \bpi_{\lambda_\ell,a_\ell}$ where $a_1,\cdots ,a_\ell$ are all distinct, then for
$i\in I$
$$\left(\bop^\pm_{i}(u)-\pi_i^\pm(u) \right)w_{\bpi}=0 \ \iff\ \
(h_i\otimes
t^r)w_{\bpi}=\left(\sum_{j=1}^\ell\lambda_j(h_i)a_j^r\right)
w_{\bpi}.$$  Let $b\in\bc^\times$ and let $\tau_bW(\bpi)$ be the
$L(\lie g)$--module obtained by pulling back $W(\bpi)$ through the
automorphism $\tau_b$ of $L(\lie g)$. The next result is standard.
\begin{lem}
\begin{enumerit}
\item Let $\bpi\in\cal P^+$. Then $W(\bpi)=\bu(L(\lie n^-))w_\bpi$, and hence we have,
$$\wt(W(\bpi)\subset\lambda_\bpi-Q^+,\ \ \dim W(\bpi)_{\lambda_\bpi}=1.$$ In particular, the module
 $W(\bpi)$ has a unique irreducible
quotient $V(\bpi)$. \item For $b\in\bc^\times$, we have
$\tau_bW(\bpi)\cong W(\bpi_b)$, where $\bpi=(\pi_i(u))_{i\in I}$ and
$\bpi_b=(\pi_i(b^{-1}u))_{i\in I}$. In particular  we have
$$W(\bpi_{\lambda,a})\cong_{\lie g}
W(\bpi_{\lambda,ab}).$$\end{enumerit}
\end{lem}

\subsection{} The modules $W(\bpi)$ were initially defined and studied in \cite{CPweyl} and a formula was conjectured for their dimension.
Parts (i) and (ii) of the next theorem were proved in
\cite{CPweyl}. Part (iii) was proved in \cite{CPweyl} in the case
of $\lie{sl_2}$, for $\lie{sl_n}$ it was proved in \cite{CL} and
for the general simply laced case in \cite{FoL}. Part (iii) can be
deduced for the general case by using results of
\cite{BN},\cite{K},\cite{Nak} for quantum affine algebras.

\begin{thm}\label{affine}
\begin{enumerit}
\item Given $\bpi=(\pi_i)_{i\in I}$ with unique decomposition
$\bpi=\prod_{k=1}^\ell \bpi_{\lambda_\ell,a_\ell}$, we have an
isomorphism of $L(\lie g)$--modules
$$W(\bpi)\cong\otimes_{k=1}^\ell W(\bpi_{\lambda_k,a_k}).$$
\item Let $V$ be any finite--dimensional $L(\lie g)$--module
generated by an element $v\in V$ such that $$L(\lie n^+)v=0,\ \
L(\lie h)v=\bc v.$$ Then there exists $\bpi\in\cal P^+$ such that
the assignment $w_\bpi\to v$ extends to a surjective homomorphism
$W(\bpi)\to V$ of $L(\lie g)$--modules. \item Let $\lambda\in P^+$
and $a\in\bc^\times$. Suppose that $\lambda=\sum_{i\in
I}m_i\omega_i$. Then
$$W(\bpi_{\lambda,a})\cong_{\lie g} \bigotimes_{i\in I}
W(\bpi_{\omega_i,1})^{\otimes m_i}.$$\end{enumerit}\hfill\qedsymbol
\end{thm}

\subsection{Annihilating ideals for $W(\bpi)$.}
The  next proposition  is implicit in \cite{CPweyl}  but since it
plays a big role in this paper we make it explicit and give a
proof.
\begin{prop}\label{annuntwist}
 Let $\bpi = \prod_{r=1}^\ell \bpi_{\lambda_r,a_r}\in \cal P^+$. There exists
an integer $N=N(\bpi)$  such that  $$\left(\lie g\otimes
\prod_{r=1}^\ell(t-a_r)^N\bc[t,t^{-1}] \right) W(\bpi)=0.$$

\end{prop}
\begin{pf} We begin by proving that for all $i\in I$
\begin{equation} \label{simplei}
x_i^-\otimes\prod_{r=1}^\ell(t-a_r)^{\lambda_r(h_i)}w_{\bpi}=0.\end{equation} Set $N_i=\lambda_\bpi(h_i)$. Using  the defining relations of
$W(\bpi)$ and Lemma \ref{Garland},
$$0=(x^+_i\otimes
t)^{N_i}(x_i^-\otimes 1)^{N_i+1}w_{\bpi} =(-1)^{N_i}\left( \mathbf x_i^-(u) \bop_i^+(u) \right)_{N_i}w_{\bpi}.$$  We also have
\[  \bop_i(u).w_\bpi =\prod_{r=1}^\ell (1-a_ru)^{\lambda_r(h_i)}.w_{\bpi}\equiv \left(\sum_{j=0}^{N_i}{p_{i,j}}u^j\right).w_{\bpi}.\]
Combining these we get$$ \left( \mathbf x_i^-(u) \bop_i^+(u)
\right)_{N_i}w_{\bpi}=\left(\sum_{j=0}^{N_i}x_i^-\otimes
p_{i,N_i-j}t^j\right)w_{\bpi}  =x_i^-\otimes
\left(\sum_{j=0}^{N_i} t^j p_{i,N_i-j} \right)w_{\bpi}=0.$$ But it
is elementary to see that $$ \sum_{j=0}^{N_i} t^j
p_{i,N_i-j}=\prod_{r=1}^\ell(t-a_r)^{\lambda_r(h_i)},$$ which
proves \eqref{simplei}. Since $\lie n^-$ is  generated by the
elements $x_i^-$, $i\in I$, it is immediate that there exists
$N\gg0$ such that \begin{equation}
\label{theta}\left(x^-_\theta\otimes
\prod_{r=1}^\ell(t-a_r)^N\right)w_{\bpi}=0.\end{equation} Since
$[\lie n^-, x^-_\theta]=0$ and $W(\bpi)\cong\bu(L(\lie
n^-))w_\bpi$ as vector spaces, we get
$$\left(x^-_\theta\otimes \prod_{r=1}^\ell(t-a_r)^N\right)W({\bpi})=0.$$ Since any
element in $\lie g$ is in the span of elements of the form $\{[x^+_{i_1}[x^+_{i_2}[\cdots[x^+_{i_k}, x^-_{\theta}],\cdots]]: i_1,\cdots ,i_k\in
I\}$, we now get $$\left(\lie g\otimes \prod_{r=1}^\ell(t-a_r)^N\bc[t,t^-1] \right) W(\bpi)=0.$$

\end{pf}
\begin{cor} Given $\bpi\in\cal P^+$ with unique decomposition $\bpi=\prod_{r=1}^\ell\bpi_{\lambda_r,a_r}\in\cal P^+$, there exists $N\in\bz_+$ such that the action of $L(\lie g)$ on
$W(\bpi)$ factors through to an action of $\lie g_{\boa, N}$ on
$W(\bpi)$ and $W(\bpi)=\bu(L(\lie n^-_{\boa,N}))w_\bpi.$\end{cor}

\section{The twisted algebras $L^\sigma(\lie g)$ and the modules $W(\bpi^\sigma)$}

\subsection{} \label{extlambda} Assume from now on that $\lie g$ is simply--laced and that
 $\sigma:\lie g\to \lie g$ is a non--trivial diagram automorphism of $\lie g$ of order $m$.
  In particular $\sigma$ induces a permutation of $I$ and $R^+$
 and we have $$\sigma(\lie g_\alpha)=\lie g_{\sigma(\alpha)},\ \ \sigma(\lie h)= \lie h,\
\sigma(\lie n^\pm)=\lie n^\pm. $$ Let $\zeta$ be a primitive
$m^{th}$ root of unity, we have  $$\lie
g=\bigoplus_{\epsilon=0}^{m-1}\lie g_\epsilon,\ \ \lie
g_\epsilon=\{x\in\lie g: \sigma(x)=\zeta^\epsilon x\}.$$ Given any
subalgebra $\lie a$ of $\lie g$ which is preserved by $\sigma$,
set $\lie a_\epsilon=\lie g_\epsilon\cap\lie a$. It is known that
$\lie g_0$ is a simple Lie algebra, $\lie h_0$ is a Cartan
subalgebra and  that  $\lie g_\epsilon$ is an irreducible
representation of $\lie g_0$ for all $0\ leq \epsilon\le m-1$. Moreover,
$$\lie n^\pm\cap \lie g_0=\lie n_0^\pm =\bigoplus_{\alpha\in R^+_{\lie g_0}}(\lie g_0)_ {_{_{\pm \alpha}}}. $$
The following table describes the various possibilities for $\lie
g$, $\lie g_0$ and the structure of $\lie g_k$ as a $\lie
g_0$--module, here $\theta_0^s$ is the highest short root of $\lie
g_0$ and $B_1=A_1$.

\[
\begin{array}{c|c|c|c|}
m & \fg & \fg_0 & {\lie g_k} \\ \hline

2 & A_{2n},  & B_n & V_{\lie g_0}( 2 \theta^s_0)  \\
2 & A_{2n-1}, \ n \geq 2 & C_n & V_{\lie g_0}(\theta_0^s)\\
2 & D_{n+1}, \ n \geq 3 & B_n & V_{\lie g_0}(\theta_0^s) \\
2 & E_6 &  F_4 &  V_{\lie g_0}(\theta_0^s) \\
3&D_4&G_2& V_{\lie g_0}(\theta_0^s)\\
\end{array}
\]
{\em From now we set $R_{\lie g}=R^+$, $R_{\lie g_0}=R_0$, the
sets $I$, $P^+$ etc. are defined similarly. The set of $\sigma$--
orbits of $I$ has  the same cardinality as $I_0$ and we identify
$I_0$ with a subset of $I$. In the case when $\lie g$ is of type
$A_{2n}$ we assume that $n\in I_0$ corresponds to the unique short
simple root of $\lie g_0$. We shall also fix $\zeta$ a primitive
$m^{th}$ root of unity. }

Suppose that  $\{y_i: i\in I\}$ is one of the sets $\{h_i: i\in
I\}$, $\{x_i^+:i\in I\}$ or $\{x_i^-:i\in I\}$ and assume that $m=2$
and that $i\ne n$ if $\lie g$ is of type $A_{2n}$. Define subsets
$\{y_{i,\epsilon}: i\in I_0, \ 0\le \epsilon\le 1\}$ of $\lie
g_\epsilon$ by
\begin{gather*} y_{i,0} =  y_i  \quad \mbox{ if } i =
\sigma(i),\qquad y_{i,0}= y_i+ y_{\sigma(i)} \quad \mbox{ if } i
\neq \sigma(i),
\\
y_{i,1}=  y_i - y_{\sigma(i)}  \quad \mbox{ if } i \neq
\sigma(i)\qquad y_{i,1}=  0 \quad \mbox{ if } i =\sigma(i),
\end{gather*}  If $\lie g$ is of type $A_{2n}$, then we set,
\begin{gather*}
 h_{n,0}= 2(h_n + h_{n+1}),\ \ x^\pm_{n,0}=\sqrt{2}(x^\pm_{n}+x^\pm_{n+1}),\\
x^\pm_{n,1}= -\sqrt{2}( x^\pm_n - x^\pm_{n+1}),\ \ h_{n,1}=h_n-h_{n+1}, \\
y^\pm_{n,1} = \mp \frac14 \left[ x^\pm_{n,0}, x^\pm_{n,1} \right]
.\end{gather*} Finally if $\frak g$ is of type $ D_4$ and $m=3$,
set,
\begin{gather*} y_{i,0} =  y_i  \quad \mbox{ if } i =
\sigma(i),\qquad y_{i,0}=\sum\limits_{j=0}^{m-1} y_{\sigma^{j}(i)}
\quad \mbox{ if } i \neq \sigma(i),
\\
y_{i,1}=y_{i,2}=  0 \quad \mbox{ if } i =\sigma(i),\\
y_{i,1}=  y_i + \zeta^2 y_{\sigma(i)} + \zeta y_{\sigma^2(i)},
\qquad y_{i,2}=  y_i + \zeta y_{\sigma(i)} + \zeta^2 y_{\sigma^2(i)}
\quad \mbox{ if } i \neq \sigma(i),\qquad  \end{gather*}

{\em In the rest of this  paper in the case when $\lie g$ is of type
$A_{2n}$, we shall only be interested in elements $\lambda\in P_0^+$
such that $\lambda(h_{n,0})\in 2\bz_+$ and we let $P_\sigma^+$
denote this subset of $P_{0}^+$. Moreover we regard $\lambda\in
P_\sigma^+$ as an element of $P^+$ as follows:
$$\lambda(h_i)=\begin{cases}
\lambda(h_{i,0}), \ \ i\in I_0,\  \mbox{ if $\lie g$\ is not of
type } A_{2n}\\
0\ \ i\notin I_0,\\
(1-\delta_{i,n}/2)\lambda(h_{i,0}), \mbox{ if $\lie g$\ is of type }
A_{2n}.
\end{cases}
$$}

\subsection{}
Let $\widetilde{\sigma}:\loopg\to \loopg$  be the automorphism
defined by extending,
\[ \widetilde{\sigma} (x\otimes t^k)= \zeta^k \sigma(x) \otimes t^k,
\]
for $x \in \frak g$, $k \in \mathbb{Z}$. Then $\tilde \sigma$ is
or order $m$ and we let $L^\sigma(\lie g)$ be the subalgebra of
fixed points of $\tilde\sigma$. Clearly,
\[L^\sigma(\lie g)\cong\bigoplus_{\ep=0}^{m-1} \lie g_\ep\otimes t^{m-\ep}\bc[t^{m}, t^{-m}].\]

\begin{lem}\label{subalg}
Let $i\in I_0$ and assume that $i\ne n$ if $\lie g$ is of type
$A_{2n}$.  The subalgebra of $L^\sigma(\lie g))$ spanned by the
elements $\{x^\pm_{i,\epsilon}\otimes t^{mk-\epsilon},
h_{i,\epsilon}\otimes t^{mk-\epsilon}: k\in\bz,  \ 0 \leq \ep \leq m-1 \}$  is canonically
isomorphic to $L(\lie{sl_2})$. If $\lie g$ is of type $A_{2n}$ the
subalgebra of $L^\sigma(\lie g))$ spanned by the elements
$\{x_{n,\epsilon}^\pm\otimes t^{2k+\epsilon},
h_{n,\epsilon}\otimes t^{2k+\epsilon}, \mp\frac
14[x^\pm_{n,0},x^\pm_{n,1}]\otimes t^{2k+1}:k\in\bz, \ 0 \leq \ep \leq m-1 \}$ is
canonically  isomorphic to $L^\sigma(\lie
{sl_3})$.\hfill\qedsymbol

\end{lem}

\subsection{Identities in  $\bu(L^\sigma(\lie g))$}
Suppose that either   $\frak g$ is not of type $A_{2n}$ and
$\alpha_i \in (R_0)_s^+$ or that $\frak g$ is of type $A_{2n}$ and
$i\ne n$.  Define power series with coefficients in
$\bu(L(^\sigma(\lie g))$ by,
\begin{gather*}\bop_{i,\sigma}^{\pm}(u) = \text{exp} \left( -
\sum_{k=1}^{\infty} \sum_{\epsilon=0}^{m-1}\frac{h_{i,\epsilon}
\otimes t^{mk-\epsilon}}{mk-\epsilon} u^{mk-\epsilon}\right),\\
\mathbf{x}^-_i(u) = \sum_{k=0}^\infty \sum_{\ep =0}^{m-1} (x^-_{i,
m-\ep} \otimes t^{mk+ \ep}) u^{mk+\ep +1},\  \qquad
\tilde{\mathbf{x}}^+_i(u) = \sum_{k=-\infty}^\infty \sum_{\ep
=0}^{m-1} (x^+_{i,m- \ep} \otimes t^{mk+ \ep})
u^{mk+\ep+1}.\end{gather*}

If $\frak g$ is not of type $A_{2n}$ and $\alpha_i \in (R_0)_l^+$,
then we set \begin{gather*}\bop_{i,\sigma}^{\pm}(u) = \text{exp}
\left( - \sum_{k=1}^{\infty} \frac{h_{i,0}\otimes t^{mk}}{k}
u^{k}\right),\\
\mathbf{x}^-_i(u) = \sum_{k=0}^\infty ( x^-_{i, 0} \otimes t^{mk}
u^{k +1})\ \qquad \tilde{\mathbf{x}}^+_i(u) = \sum_{k=0}^\infty
(x^+_{i, 0} \otimes t^{mk} )u^{k+1}.\end{gather*} Finally, if
$\lie g$ is of type $A_{2n}$ and $i=n$ we
have,\begin{gather*}\bop_{n,\sigma}^{\pm}(u) = \text{exp} \left(-
\sum_{k=1}^{\infty} \frac{h_{n,0}/{2}\otimes t^{2k}}{2k} u^{2k}+
\sum_{k=1}^{\infty} \frac{ h_{n,1}\otimes t^{2k-1}}{2k-1}
u^{2k-1}\right),\\
\mathbf{x}^-_n(u) = \sum_{k=0}^\infty \sum_{\ep =0}^{m-1} (x^-_{n,
\ep} \otimes t^{mk+ \ep}) u^{mk+\ep +1},\  \qquad
\tilde{\mathbf{x}}^+_n(u) = \sum_{k=-\infty}^\infty \sum_{\ep
=0}^{m-1} (x^+_{n, \ep} \otimes t^{mk+ \ep})
u^{mk+\ep+1}.\end{gather*}

\begin{lem} Let $r\in\bz_+$.
\begin{enumerit}
\item If $\frak g$ is not of type $A_{2n}$ and $\alpha_i \in
(R_0)_s^+$ or $\frak g$ is of type $A_{2n}$ and $\alpha_i \in
(R_0)_l^+$, we have
$$(x_{i,1}^+\otimes t)^{(r)}_{} (x_{i,0}^-\otimes 1)^{ (r+1)} =(-1)^r
\left( \mathbf x^-_{i}(u)   \bop_{i,\sigma}^+(u) \right)_{r+1}
\mod \bu(L^\sigma(\lie{g}))\widetilde{\mathbf x_i}^+(u).$$ \item
If $\frak g$ is not of type $A_{2n}$ and $\alpha_i\in (R_0)_l^+$,
\begin{gather*}
 (x_{i,0}^+\otimes t^2)^{(r)} (x_{i,0}^-\otimes 1)^{ (r+1)} =(-1)^r \left(
\mathbf x^-_{i}(u)   \bop_{i,\sigma}^+(u) \right)_{r+1} \mod
\bu(L^\sigma(\lie{g}))\widetilde{\mathbf x_i}^+(u).\end{gather*}
\item If $\frak g$ is of type $A_{2n}$, we have\\

 \begin{enumerate}
\item[(a)]  $ (x_{n,0}^+\otimes 1)^{ (2r-1)} (y_{n,1}^-\otimes
t)^{(r)} =- \left(  \mathbf x_n^-(u)
 \bop_{n,\sigma}^+(u) \right)_r \mod
\bu(L^\sigma(\lie{\frak g}))\widetilde{\mathbf{x}}_n^+(u),$\\
\item[(b)] $ (x_{n,0}^+\otimes 1)^{ (2r)} (y_{n,1}^-\otimes
t)^{(r)} =- \left( \bop_{n,\sigma}^+(u) \right)_r \mod
\bu(L^\sigma(\lie{\frak g}))\widetilde{\mathbf x}_n^+(u),$\\
where
$$y_{n,1}^-=\frac14\left[x^-_{n,0}, x^- _{n,1}\right].$$
\end{enumerate}
\end{enumerit}
\end{lem}
\begin{pf} Parts (i) and (ii) are immediate consequences of Lemma
\ref{Garland} and Lemma \ref{subalg}. Part (iii) is deduced from
 \cite{Mitz},  \cite[Lemma 5.36]{FV},  exactly as (i) and (ii) were
deduced from Garland in \cite{CPweyl}.
\end{pf}

\subsection{The monoid $\cal P_\sigma^+$} Let  $( \ , )$ be the form on $\lie h_0^*$ induced by
the Killing form of $\lie g_0$ normalized so that
$(\theta_0,\theta_0)=2$. For $i\in I_0$ and $a\in\bc^\times$,
$\lambda\in P_0^+$
 and  $\lie g$  not of type $A_{2n}$ let
$$\bpi^\sigma_{i,a}=((1-a^{(\alpha_i,\alpha_i)}u)^{\delta_{ij}}: j\in
I_0),\qquad \bpi^\sigma_{\lambda,a}=\prod_{i\in
I_0}\left(\bpi^\sigma_{i,a }\right)^{\lambda(h_i)},$$ while if
$\lie g$ is of type $A_{2n}$ we set for $i\in I_0$,
$a\in\bc^\times$, $\lambda\in P^+_\sigma$,
$$\bpi^\sigma_{i,a}=((1-au)^{\delta_{ij}}:
j\in I_0),\qquad \bpi^\sigma_{\lambda,a}=\prod_{i\in
I_0}\left(\bpi^\sigma_{i,a
}\right)^{(1-\frac12\delta_{i,n})\lambda(h_i)}.$$
 Let
$\cal P_\sigma^+$ be the monoid generated by the elements
$\bpi^\sigma_{\lambda,a}$. Define a  map $\cal P^+_\sigma\to
P_\sigma^+$ by $$\lambda_{\bpi^\sigma}=\sum_{i\in
I_0}(\deg\pi_i)\omega_i,$$ if $\lie g$ is not  of type $A_{2n}$ and
$$\lambda_{\bpi^\sigma}=\sum_{i\in
I_0}(1+\delta_{i,n})(\deg\pi_i)\omega_i,$$ if $\lie g$ is  of type
$A_{2n}$. It is clear that any $\bpi^\sigma \in \cal P^+_\sigma$ can be
written (non--uniquely) as product
$$\bpi^\sigma=\prod_{k=1}^\ell\prod_{\ep=0}^{m-1}\bpi^\sigma_{\lambda_{k,\ep},
\zeta^\ep a_k},$$ where $\boa=(a_1,\cdots ,a_\ell)$ and $\boa^m$
have distinct coordinates. We call any such expression a standard
decomposition of $\bpi^\sigma$.

\subsection{The set $\boi(\bpi^\sigma)$} Given $\lambda=\sum_{i\in I} m_i\omega_i\in P^+$
and $0\le \ep\le m-1$, define elements $\lambda(\epsilon)\in
P^+_\sigma$ by, \begin{gather*}\lambda(0)=\sum_{i\in I_0}
m_i\omega_i,\ \ \lambda(1)=\sum_{i\in I_0: \sigma(i)\ne
i}m_{\sigma(i)}\omega_i,\ \ \ {\text{if}}\ \ m=2\ \
 {\text{and}}\ \ \lie g \mbox{ not of type } A_{2n}\\
\lambda(0)=\sum_{i\in I_0} (1 + \delta_{i,n})m_i\omega_i,\ \
\lambda(1)=\sum_{i\in I_0: \sigma(i)\ne i}(1 +
\delta_{\sigma(i),n})m_{\sigma(i)}\omega_i,\ \  {\text{if}}\ \
m=2\ \
 {\text{and}}\ \ \lie g \mbox{  of type } A_{2n}\\
  \lambda(0) = m_1\omega_1 + m_2\omega_2, \ \
\lambda(1) = m_3 \omega_1, \ \ \lambda(2) = m_4 \omega_1, \ \ \
{\text{if}}\ \ m=3.\end{gather*}

Define a map $\bor:\cal P^+\to\cal P_\sigma^+$  as follows. Given
$\bpi\in\cal P^+$ write
$$\bpi=\prod_{k=1}^\ell\bpi_{\lambda_k,a_k}, \ \ a_k\ne a_p, \ \ 1\le k\ne p\le \ell,
$$ and
set
$$\bor(\bpi)=\prod_{k=1}^\ell\prod_{\epsilon=0}^{m-1}\bpi^\sigma_{\lambda_k(\epsilon),\zeta^\epsilon
a_k}.$$ Note that $\bor$ is well defined since the choice of
$(\lambda_k,a_k)$ is unique and set
$$\boi(\bpi^\sigma)=\{\bpi\in\cal P^+:\bor(\bpi)=\bpi^\sigma\}.$$
We now  give an  explicit description of the set
$\boi(\bpi^{\sigma})$. Recall that given $\lambda\in P_\sigma^+$, we
also regard $\lambda\in P^+$ as in Section \ref{extlambda}. In
addition, define $\sigma(\om_i)=\om_{\sigma(i)}$ for $i\in I$.
\begin{lem}\label{boi decomp}\begin{enumerit}\item  Let $i\in I_0$ and $a\in
\bc^{\times}$. We
have,
$$\boi(\bpi^{\sigma}_{\om_i,a}) = \{ \bpi_{\sigma^{r}(\om_i),\zeta^{m-r}a} \, |
\, 0\leq r < m \},$$
and for $A_{2n}^2$ and $i=n$,
$$\boi(\bpi^{\sigma}_{2\om_n,a}) = \{ \bpi_{\om_n,a} \, , \bpi_{\om_{n+1},-a}
\}$$
\item Let $\bpi^{\sigma} = \prod_{k=1}^\ell\prod_{\ep=0}^{m-1}\prod_{i\in
I_0}(\bpi^\sigma_{\om_{i},\zeta^\ep a_k})^{m_{k,\ep,i}}$ be a decomposition of
$\bpi^\sigma$ into linear factors for $\lie g$ not of type $A_{2n}$. Then
$$
\boi(\bpi^{\sigma}) = \prod_{k=1}^\ell\prod_{\ep=0}^{m-1}\prod_{i\in I_0} \{
\bpi_{\sigma^{r}(\om_i),\zeta^{m-r+\ep}a_k} \, | \, 0\leq r < m
\}^{m_{k,\ep,i}}$$
where the product of the sets is understood to be the set of products of
elements of the sets.\\ In the case of $A_{2n}^{(2)}$, let $\bpi^{\sigma} =
\prod_{k=1}^\ell\prod_{\ep=0}^{1}\prod_{i\in I_0}(\bpi^\sigma_{(1 +
\delta_{i,n})\om_{i},\zeta^\ep a_k})^{m_{k,\ep,i}}$ be a decomposition of
$\bpi^\sigma$ into linear factors. Then
$$
\boi(\bpi^{\sigma}) = \prod_{k=1}^\ell\prod_{\ep=0}^{2}\prod_{i\in I_0} \{
\bpi_{\sigma^{r}(\om_i),\zeta^{2-r+\ep}a_k} \, | \, 0\leq r < 2
\}^{m_{k,\ep,i}}$$
\item In particular, we have
\[ \prod_{k=1}^\ell \bpi_{\mu_k, a_k} =
\prod_{k=1}^\ell\prod_{\ep=0}^{m-1}\prod_{i\in I_0} \bpi_{\sigma^\ep (\om_i),
a_k}^{m_{k, \ep, i}} \in \boi(\bpi^\sigma),\]
where $\mu_k = \sum_{\ep = 0}^{m-1}\sum_{i \in I_0} m_{k, \ep,
i}\sigma^{\ep}(\om_i)$ and $a_i^m \neq a_j^m$.
\end{enumerit}
\end{lem}

\begin{pf}
The first statement is trivially checked, noting that if $i$ is a fixed point
of $\sigma$, then $\bpi^{\sigma}_{\om_i,a} = \bpi^{\sigma}_{\om_i,\zeta^r a}$
for $0 \leq r < m$. The other statements follow immediately from the first one.
\end{pf}
From here on we shall assume that, unless otherwise noted, the element $\bpi
\in \boi(\bpi^\sigma)$ chosen is of the form given in (iii) of the lemma.

\subsection{The modules $W(\bpi^\sigma)$, $V(\bpi^\sigma)$}

 Given $\bpi^\sigma=(\pi_{i,\sigma})_{i\in I_0}\in\cal
P^+_\sigma$, the Weyl module $W(\bpi^\sigma)$ is the
$\bu(L^\sigma(\frak g))$-module generated by an element
$w_{\bpi^\sigma}$ with relations:
\begin{gather*} L^\sigma(\lie
n^+)w_{\bpi^\sigma}=0,\ \ hw_{\bpi}=\lambda_\bpi(h)
w_{\bpi^\sigma},\ \
(x^-_{i,0})^{\lambda_\bpi(h_i)+1}w_{\bpi^\sigma}=0,\\
\\ \left(\bop^\pm_{i,\sigma}(u)-\pi_{i,\sigma}^{\pm}(u)\right)w_{\bpi^\sigma}=0,
\end{gather*}  for all $i\in I_0$ and $h\in\lie h_0$.
If $\bpi^\sigma=\prod_{k=1}^\ell \bpi_{\lambda_k,a_k}^\sigma\in\cal
P_\sigma^+$, it is not hard to see that for $i\in I_0$, we have if
$\frak g$ not of type $A_{2n}$,
\begin{equation}\label{equiv1}\left(\bop^\pm_{i,\sigma}(u)-\pi_{i,\sigma}^\pm(u) \right)w_{\bpi^\sigma}=0
\ \iff \ \ (h_{i,\ep}\otimes
t^{mk-\ep})w_{\bpi^\sigma}=\sum_{j=1}^\ell\lambda_j(h_{i,0})a_j^{mk-\ep}
w_{\bpi^\sigma},\end{equation} and for $\frak g$ of type $A_{2n}$,
\begin{equation}\label{equiv2}\left(\bop^\pm_{i,\sigma}(u)-\pi_{i,\sigma}^\pm(u) \right)w^\sigma_{\lambda,a}=0 \ \iff \ \
(h_{i,\ep}\otimes t^{mk-\ep})w_{\bpi^\sigma}=\sum_{j=1}^\ell(1-\frac
12\delta_{i,n})\lambda_j(h_{i,\ep})a_j^{mk-\ep}
w_{\bpi^\sigma}.\end{equation}

\subsection{} For $b\in\bc^\times$ we have $\tau_b(L^\sigma(\lie
g))\subset L^\sigma(\lie g)$ and we let  $\tau_bW(\bpi^\sigma)$ be
the $L^\sigma(\lie g)$--module obtained by pulling back
$W(\bpi^\sigma)$ through $\tau_b$. The next result is proved by
standard methods.
\begin{lem}\label{standlemtwisted}
\begin{enumerit}
\item Let $\bpi^\sigma\in\cal P^+_\sigma$. Then $W(\bpi^\sigma)=\bu(L^\sigma(\lie n^-))w_\bpi^{\sigma}$, and hence we have,
$$\wt(W(\bpi^\sigma)\subset\lambda_{\bpi^\sigma}-Q_0^+,\qquad \dim W(\bp^\sigma)_{\lambda_{\bpi^\sigma}}=1.$$ In particular, the module
 $W(\bpi^\sigma)$ has a unique irreducible
quotient $V(\bpi^\sigma)$. \item For $b\in\bc^\times$, we have
$\tau_bW(\bpi^\sigma)\cong W(\bpi_b^\sigma)$, where
$\bpi^\sigma=(\pi_i(u))_{i\in I}$ and
$\bpi^\sigma_b=(\pi_i(b^{-1}u))_{i\in I}$. In particular  we have
$$W(\bpi^\sigma_{\lambda,a})\cong_{\lie g_0}
W(\bpi^\sigma_{\lambda,ba}).$$\end{enumerit}\hfill\qedsymbol
\end{lem}

\subsection{The main theorem} In the rest of  this paper we shall prove the
following result.
\begin{thm}\label{twisted}
\begin{enumerit} \item Let $\bpi^\sigma\in\cal P^+_\sigma$. For all  $\bpi\in\boi(\bpi^{\sigma})$,
we have
$$W(\bpi^\sigma)\cong_{L^\sigma(\lie g)} W(\bpi),\qquad V(\bpi^\sigma)\cong_{L^\sigma(\lie g)}
V(\bpi).$$\\
\item Let $\bpi^\sigma\in\cal P^+_\sigma$ and assume that $
\prod_{k=1}^\ell\prod_{\epsilon=0}^{m-1}\bpi^\sigma_{\lambda_{k,\epsilon},
\zeta^\epsilon a_k}\in \cal P^+_\sigma$ is a standard decomposition
of $\bpi$. As $L^{\sigma}(\lie g)$--modules, we have
$$
W(\bpi^{\sigma}) \cong \bigotimes_{k=1}^{\ell}
W(\prod_{\epsilon=0}^{m-1}\bpi^\sigma_{\lambda_{k,\epsilon},
\zeta^\epsilon a_k}).$$\\ \item  Suppose that
$\prod_{\epsilon=0}^{m-1}\bpi^\sigma_{\lambda_{\epsilon},
\zeta^{\epsilon} a}\in \cal P^+_\sigma$. Then
$$W(\prod_{\epsilon=0}^{m-1}\bpi^\sigma_{\lambda_{\epsilon},
\zeta^{\epsilon} a}) \cong_{\lie g_0}
\bigotimes_{\epsilon=0}^{m-1} W(\bpi^\sigma_{\lambda_{\epsilon},
\zeta^{\epsilon} a}).
$$
\item Let $\lambda=\sum_{i\in I_0}m_i\omega_i\in P_{\sigma}^+$ and
$a\in\bc^\times$. We have for $\lie g$ not of type $A_{2n}$
$$W(\bpi^{\sigma}_{\lambda,a})\cong_{\lie g_0}
\bigotimes_{i=1}^{n} W(\bpi^{\sigma}_{\omega,1})^{\otimes m_i}$$ and
for $\lie g$ of type $A_{2n}$
$$W(\bpi^{\sigma}_{\lambda,a})\cong_{\lie g_0} W(\bpi^{\sigma}_{2\omega_n,1})^{\otimes \frac{m_n}{2}} \otimes \bigotimes_{i =1}^{n-1}
W(\bpi^{\sigma}_{\omega_i,1})^{\otimes m_i}.$$

\item Let $V$ be any finite--dimensional $L^\sigma(\lie g)$--module generated
by an element $v\in V$ such that $$L^\sigma(\lie n^+)v=0,\ \
L^\sigma(\lie h)v=\bc v.$$ Then there exists $\bpi^\sigma\in\cal
P^+_\sigma$ such that the assignment $w_{\bpi^\sigma}\to v$ extends
to a surjective homomorphism $W(\bpi^\sigma)\to V$ of $L^\sigma(\lie
g)$--modules.

\end{enumerit}\hfill\qedsymbol
\end{thm}

\section{Proof of Theorem \ref{twisted}}

\subsection{Annihilating ideals for  $W(\bpi^{\sigma})$.}

\begin{prop}\label{annmulti}
 Let $\bpi^\sigma = \prod_{r=1}^\ell \bpi^\sigma_{\lambda_r,a_r}\in \cal P^+_\sigma$.  There exists
an integer $N=N(\bpi)$  such that  $$
\left(\bigoplus_{\epsilon=0}^{m-1}(\lie g_\epsilon\otimes
t^{m-\epsilon}\prod_{r=1}^{\ell}(t^m-a_r^m)^N\bc[t^m,t^{-m}])
\right)W(\bpi^\sigma)=0 .$$

\end{prop}
\begin{proof}  The subalgebra
$L^m(\lie g_0)=\lie g_0\otimes \bc[t^m, t^{-m}]$ is canonically
isomorphic to $L(\lie g_0)$. It follows from the defining relations
that $$L^m(\lie n^+_0) w_{\bpi^\sigma}=0, \qquad\ (h_0\otimes
t^{mk})w_{\bpi^\sigma}=\left(\sum_{r=1}^\ell\lambda_r(h_0)a_r^{mk}
\right)w_{\bpi^\sigma},$$ and hence, $\bu(L^m(\lie
g_0))w_{\bpi^\sigma}$ is a quotient of the $L(\lie g_0)$--module
$W_{\lie {g_0}}(\bpi_m)$ where
$$\bpi_m=\prod_{r=1}^\ell\bpi_{\lambda_r,a_r^m}.$$ It follows from \eqref{theta} that
\begin{equation}\label{theta0}(x^-_{\theta_0}\otimes\prod_{r=1}^m(t^m-a_r^m)
)w^\sigma_{\bpi}=0,\end{equation} for some $N\in\bz_+$, where
$\theta_0\in R_0^+$ is the highest root in $R_0^+$.

 Assume first  that $\lie g$ is not of type
$A_{2n}$, then \begin{equation}\label{reln}[x^-_{\theta_0},
L^\sigma(\lie n^-)]=0,\ \qquad\ [\lie h_\epsilon, \lie g_0]=\lie
g_\epsilon,\ \ 0\le \epsilon\le m-1. \end{equation} The first
equality in \eqref{reln} gives $\left(x^-_{\theta_0}\otimes
\prod_{r=1}^\ell (t^m-a_r^m)^N\right)W(\bpi^\sigma)=0, $. One
deduces now as in the untwisted case that
$$\left(\lie g_0\otimes \left(\prod_{r=1}^\ell
(t^m-a_r^m)^N\right)\bc[t^m, t^{-m}]\right) W(\bpi^\sigma)=0.$$
Applying $\lie h_\ep\otimes t^{m-\ep}$ to the preceding equation and
using the second equality in \eqref{reln} gives
$$\left(\lie g_\epsilon\otimes t^{m-\epsilon}\left(\prod_{r=1}^\ell
(t^m-a_r^m)\right)\bc[t^m, t^{-m}]\right) W(\bpi^\sigma)=0,$$ for
all $0\le \epsilon\le m-1$ and the result  is proved.

Assume now that $\lie g$ is of type $A_{2n}$. This time, we use the
fact that $$(x^-_{n,\epsilon}\otimes
t^\epsilon\bc[t^2,t^{-2}])w_\bpi\in\bu(L^\sigma(\lie h\oplus\lie
n^+))(x^-_{\theta_0}\otimes \bc[t^2,t^{-2}])w_\bpi$$ together with
\eqref{theta0} to conclude that
 $$(x^-_{n,\epsilon}\otimes
\prod_{r=1}^\ell t^\epsilon(t^2-a_r^2)^N)w_{\bpi^\sigma}=0. $$
Hence
$$ \left([x^-_{\theta_0},x_{n,1}^-]\otimes
\prod_{r=1}^\ell t(t^2-a_r^2)^N)\right)w^\sigma_{\bpi}=0,$$ for some
$N\gg0$. Since the element $[x^-_{\theta_0},x_{n,1}^-]\in\lie g_1$
generates $\lie g_1$ as a $\lie g_0$--module and $[\lie
n^-,[x^-_{\theta_0},x_{n,1}^-]]=0$, we can now prove by similar
arguments that for some $N\gg0$,
 $$
(\lie g_1\otimes\prod_{r=1}^\ell t(t^2-a_r^2)^N)W(\bpi^\sigma)=0.$$
Next, using the fact that  $[x^-_{\theta_0}, \lie n^-_1]=\bc
[x^-_{\theta_0},x_{n,1}^-]$, we  get
$$(x^-_{\theta_0}\otimes \prod_{r=1}^\ell(t^2-a_r^2)^N)W(\bpi^\sigma)=0,$$ which
finally gives
$$\left(\lie g_0\otimes \prod_{r=1}^\ell(t^2-a_r^2)^N\bc[t^2, t^{-2}]\right) W(\bpi^\sigma)=0,$$ and
completes the proof.

\end{proof}
Given positive integers $\ell,N\in\bz_+$, $\boa=(a_1,\cdots
,a_\ell)\in(\bc^\times)^\ell$ and a subalgebra $\lie a$ of $\lie g$
such that $\sigma(\lie a)\subset\lie a$, let
\begin{equation}\label{gsan}\lie a^\sigma_{\boa,N}=L^\sigma(\lie
g)/\oplus_{\epsilon=0}^{m-1}(\lie a_\epsilon\otimes
t^\epsilon\prod_{k=1}^\ell
(t^m-a_k)^N\bc[t^m,t^{-m}]).\end{equation}

\begin{cor} Let $\bpi^\sigma = \prod_{r=1}^\ell \bpi^\sigma_{\lambda_r,a_r}\in \cal P^+_\sigma$
 be a standard decomposition of $\bpi^\sigma$ and set $\boa=(a_1,\cdots ,a_\ell)$
 There exists $N>>0$ such that
$$ W(\bpi^\sigma)=\bu((\lie n^-_{\boa^m,N})^\sigma)w_{\bpi^\sigma}$$
 \end{cor}

\subsection{}  \begin{prop}\label{findim} For all  $\bpi^\sigma\in\cal
P^+_\sigma$, the $L^\sigma(\lie g)$--module $W(\bpi^\sigma)$ is
finite--dimensional.
\end{prop}
\begin{pf} Let $u\in W(\bpi^\sigma)$ and write
$u=yw_{\bpi^\sigma}$ for some $y\in \bu(L^\sigma(\lie n^-))$. The
adjoint action of the subalgebras $\lie n^\pm_0$ on $L^\sigma(\lie
g)$ and hence on $\bu(L^\sigma(\lie g))$ is nilpotent. Using the
defining relations we get immediately that for some $r=r(u)>0$, we
have
$$(x_\alpha^\pm\otimes 1)^r u=0,\ \ \forall\ \ \alpha\in R_0^+.$$
This implies that $\bu(\lie g_0)u$ is a finite--dimensional $\lie
g_0$--submodule of $W(\bpi^\sigma)$, and hence $W(\bpi^\sigma)$ is
isomorphic to a direct sum of $\lie g_0$--modules. Write,
$$W(\bpi^\sigma)=\bigoplus_{\eta\in
Q_0^+}W(\bpi^\sigma)_{\mu},$$ where $W(\bpi^\sigma)_{\mu}=\{u\in
W(\bpi^\sigma): h u=\mu(h)u,\ \ \forall\ \ h\in\lie h_0\}.$ The
representation theory of $\lie g_0$ now  implies that
$$W(\bpi^\sigma)_{\mu}\ne 0\iff
W(\bpi^\sigma)_{w(\mu)}\ne 0,\ \ \forall\  w\in W_0.$$ Since
$W(\bpi^\sigma)_{\mu}=0$ unless $\mu\in\lambda-Q_0^+$ and the
number of elements in $P_0^+$ with this property is finite we get
that $W(\bpi^\sigma)_{\nu}=0,$ for all but finitely many $\nu\in
P_0^+$. The proposition follows if we prove that
$\dim(W(\bpi^\sigma)_\nu)<\infty$ for all $\nu\in P_0^+$.

 Choose $\boa$ and $N$ as in
Corollary  \ref{annmulti}. Then  $$W(\bpi^\sigma)_\nu =\bu((\lie
n^-_{\boa,N})^\sigma)_{\lambda_\bpi-\nu}w_{\bpi^\sigma}$$ where
$$\bu((\lie
n^-_{\boa,N})^\sigma)_{\lambda_\bpi-\nu}=\{y\in\bu((\lie
n^-_{\boa,N})^\sigma)_{\lambda_\bpi-\nu}: [h,y]=(\lambda_\bpi -
\nu)(h)y,\ \ \forall \ h\in\lie h_0\}.$$ Since this subspace is
finite--dimensional  it follows that
$\dim(W(\bpi^\sigma)_\nu)<\infty$ as required.
\end{pf}
\subsection{} Let $N\in\bz_+$ and $\boa\in(\bc^\times)^\ell$.
 The inclusion $\iota:
L^\sigma(\lie g)\to L(\lie g)$ obviously induces a Lie algebra map
$\iota_{\boa,N}:\lie g^\sigma_{\boa^m,N}\to\lie g_{\boa,N}$, where
$\boa^m=(a_1^m,\cdots ,a_\ell^m)$. The following proposition will
play  a crucial role in the proof of Theorem \ref{twisted}.
\begin{prop}\label{iso} Let $\boa\in(\bc^\times)^\ell$ be such that
$\boa$ and $\boa^m$ have distinct coordinates. For all $N\in\bz_+$
we have an isomorphism of Lie algebras,
 $$\lie g_{\boa,N}\cong \bigoplus\lie
g^{\sigma}_{a_i^m,N}\cong\lie g^\sigma_{\boa^m, N}$$ for all $N\in\bz_+$. In particular, the composite map $L^\sigma(\lie g)\to L(\lie g)\to
\lie g_{\boa,N}$ is  surjective.
\end{prop}
\begin{pf} The proof that $$\bigoplus\lie
g^{\sigma}_{a_i^m,N}\cong\lie g^\sigma_{\boa^m, N}$$ is an obvious modification  of the one given in Lemma \ref{chinun} which also shows now
that it is  sufficient to prove the proposition when $\ell=1$. For this, let $a\in\bc^\times$ and  $f =t^\epsilon g$ where
$g\in\bc[t^m,t^{-m}]$. Then, $$f\in (t-a)^N\bc[t,t^{-1}]\ \ \iff\ \ \ f\in t^{\epsilon}(t^m-a^m)^N\bc[t^m,t^{-m}],$$ which proves that
$\iota_{a,N}$ is injective. The proposition follows by noting that
$$\dim\lie g^\sigma_{a^m,N}=\dim\lie g_{a,N}=N\dim\lie g.$$

\end{pf}

\subsection{} We note  some elementary observations which we use without further comment.
Any $\lie g_{\boa, N}$--module  (resp. $\lie g_{\boa,N}^\sigma$) is obviously a $L(\lie g)$--module (resp. $L^\sigma(\lie g)$--module). Moreover
if  $\boa\in(\bc^\times)^\ell$ is such that $\boa$ and  $\boa^m$ have distinct coordinates then for all $N\in\bz_+$, any $\lie g_{\boa,
N}$--module $V$ is also a $\lie g_{\boa^m,N}$--module and we write it as $V_{\lie g^\sigma_{\boa^m,N}}$. Similarly if we start with a $\lie
g^\sigma_{\boa^m,N}$--module $V$ we get a $\lie g_{\boa,N}$--module which we write as $V_{\lie g_{\boa,N}}$. Note also that  if $V$ is an  $\lie
g_{\boa, N}$--module, then
\begin{equation}\label{inv}  (V_{\lie g^\sigma_{\boa^m,N}})_{_{\lie
g_{\boa,N}}}\cong_{\lie g_{\boa,N}} V, \qquad  (V_{\lie g^\sigma_{\boa^m,N}})_{_{ \lie g_{\boa,N}}}\cong_{L(\lie g)} V.\end{equation}
\subsection{} \begin{lem}\label{wsigiso} Let $\bpi^\sigma\in\cal P^+_\sigma$,
and assume that $\bpi\in\boi(\bpi^\sigma)$. \begin{enumerit} \item  There exists $\ell, N\in\bz_+$ and $\boa\in(\bc^\times)^\ell$  with  $\boa$
and $\boa^m$ having  distinct coordinates such  that $W(\bpi)$ and $W(\bpi^\sigma)$ are modules for both  $\lie g_{\boa,N}$ and $\lie
g_{\boa^m,N}$.
\item In particular,
$$W(\bpi)_{\lie g^\sigma_{\boa^m,N}}=\bu(\lie g^\sigma_{\boa^m,N})w_\bpi,$$ and
  $V(\bpi)_{\lie g^\sigma_{\boa^m,N}}$ is an irreducible $\lie g^\sigma_{\boa^m,N}$--module.
\end{enumerit}
\end{lem}
\begin{pf} Let $\bpi=\prod_{k=1}^\ell\bpi_{\lambda_k,a_k}$, where $\boa=(a_1,\cdots
,a_k)$ and $\boa^m$ have distinct coordinates.  Proposition
\ref{annuntwist} implies that $W(\bpi)=\bu(\lie g_{\boa,N})w_\bpi$.
Using proposition \ref{iso} we see that $W(\bpi)$ is also module for
$\lie g^\sigma_{\boa^m,N}$ and so we get $$W(\bpi)_{\lie
g^\sigma_{\boa^m,N}}=\bu(\lie g^\sigma_{\boa^m,N})w_\bpi.$$
Similarly Proposition \ref{annmulti} implies that $W(\bpi^\sigma)$
is a module for $\lie g^\sigma_{\boa^m,N}$ and hence for $\lie
g_{\boa,N}$. Since $V(\bpi)$ is an irreducible module for $\lie
g_{\boa, N}$, it follows that it is also irreducible as a ${\lie
g^\sigma_{\boa^m,N}}$--module and the proposition is proved.
\end{pf}

The following proposition proves (i) of Theorem \ref{twisted}.

\begin{prop} \label{quotients}  Let $\bpi^\sigma\in\cal P^+_\sigma$,
$\bpi\in\boi(\bpi^\sigma)$.

\begin{enumerit}
\item Regarded as $L^\sigma(\lie g)$--module
 $W(\bpi) $ is a
 quotient of $W(\bpi^\sigma)$ and hence $$V(\bpi)\cong_{L^\sigma(\lie g)} V(\bpi^\sigma).$$
\item For $N\gg0$, the $L^\sigma(\lie g)$--module structure on $W(\bpi^\sigma
)$ (resp. $V(\bpi^\sigma )$)   extends to an $L(\lie g)$--module
action on $W(\bpi^\sigma )$  (resp. $V(\bpi^\sigma )$).
\item The module $W(\bpi^\sigma)_{\lie g_{\boa,N}}$ is a $L(\lie
g)$--module quotient of $W(\bpi)$.

\end{enumerit}
\end{prop}
\begin{pf} Using
 \eqref{equiv1}, \eqref{equiv2} and the fact that $\bor(\bpi)=\bpi^\sigma$, we  see that $w_\bpi$ satisfies the
defining relations of $W(\bpi^\sigma)$. Part (i) follows if we prove
that $W(\bpi)=\bu(L^\sigma(\lie g))w_\bpi.$ But this is true because
proposition \ref{iso} and proposition \ref{wsigiso} prove that there
exists $\boa\in(\bc^\times)^\ell$ such that
$$W(\bpi)_{\lie g^\sigma_{\boa^m,N}}=\bu(\lie g^\sigma_{\boa^m,N})w_\bpi=\bu( L^\sigma(\lie
g))w_\bpi.$$ It now follows that $V(\bpi)_{\lie g_{\boa^m,N}}$ is
the irreducible quotient of $W(\bpi^\sigma)$ and hence is isomorphic
to $V(\bpi^\sigma)$ as $L^\sigma(\lie g)$--modules.

To prove (ii), note that  that we have  a surjective homomorphism
of Lie algebras $$\bop: L(\lie g)\to\lie g_{\boa,N}\to\lie
g^\sigma_{\boa^m,N}, $$ such that the restriction of $\bop$ to
$L^\sigma(\lie g)$ is just the canonical surjection.
 Moreover
$$\bop(L(\lie n^\pm))\subset(\lie n^\pm)^\sigma_{\boa^m,N},\ \
\bop(L(\lie h))\subset\lie h^\sigma_{\boa^m,N},$$ and hence
$$L(\lie n^+)w_{\bpi^\sigma}=0,\ \ L(\lie h)w_{\bpi^\sigma}=\bc
w_{\bpi^\sigma}.$$ Since $\dim(W(\bpi^\sigma))<\infty$, it follows
from Theorem \ref{affine}(i) that $W(\bpi^\sigma)_{\lie
g_{\boa,N}}$ is a quotient of $W(\tilde\bpi)$ for  some
$\tilde\bpi\in\cal P^+$. Since the module $W(\tilde\bpi)$ has an
unique irreducible quotient $V(\tilde\bpi)$, part (iii) follows if
we prove that
$$V(\bpi)\cong_{L(\lie g)} V(\bpi^\sigma)_{\lie g_{\boa,N}}.$$ But
this follows from part (i) and \eqref{inv} and part (iii) is now
proved.
\end{pf}

\subsection{}
The next proposition proves part (ii) of Theorem \ref{twisted}.
\begin{prop}
Let $\bpi^\sigma=
\prod_{k=1}^\ell\prod_{\epsilon=0}^{m-1}\bpi^\sigma_{\lambda_{k,\epsilon},
\zeta^\epsilon a_k}\in \cal P^+_\sigma$ and assume that $\boa$ and
$\boa^m$ have distinct coordinates. As $L^{\sigma}(\lie
g)$--modules, we have
$$
W(\bpi^{\sigma}) \cong \bigotimes_{k=1}^{\ell}
W\left(\prod_{\epsilon=0}^{m-1}\bpi^\sigma_{\lambda_{k,\epsilon},
\zeta^\epsilon a_k}\right).
$$
\end{prop}
\begin{proof} For $1\le k\le \ell$, set $$\bpi^\sigma_k=\prod_{\epsilon=0}^{m-1}\bpi^\sigma_{\lambda_{k,\epsilon},
\zeta^\epsilon a_k}.$$ It is checked easily that the element
$\otimes_{k=1}^\ell{w_{\bpi^\sigma_k}}$ satisfies the defining
relations of $w_{\bpi^\sigma}$ and hence we have an $L^\sigma(\lie
g)$--module map, $$ \eta: W(\bpi^\sigma)\to \bigotimes_{k=1}^\ell
W(\bpi^\sigma_k).$$ The proposition follows if we prove that this
map is surjective. For then, taking  $\bpi\in \boi(\bpi^\sigma)$
and $\bpi_k\in\boi(\bpi_k^\sigma)$,  we have
$$\dim W(\bpi)=\dim W(\bpi^\sigma)\ge \prod_{k=1}^\ell\dim
W(\bpi^\sigma_k)=\prod_{k=1}^\ell\dim W(\bpi_k)=\dim W(\bpi),$$
where the first equality uses  part (i) of Theorem \ref{twisted}
and the  last equality follows from Theorem \ref{affine}(iii). To
prove that $\eta$ is surjective choose $N>>0$ so that
$W(\bpi^\sigma)$ is a module for $\lie g_{\boa^m,N}$ and also so
that for all $1\le k\le \ell$ the algebra $\lie
g^\sigma_{\boa_k^m,N}$ acts on $W(\bpi_k^\sigma)$ where
$\boa_k^m=\{a_k^m\}$ and we have
$$W(\bpi_k^\sigma)=\bu(\lie g^\sigma_{\boa_k^m,N})w_{\bpi^\sigma_k}.$$ On the other
hand by Proposition \ref{iso} we have
 $$  \lie
g^{\sigma}_{\boa^m,N} \cong \bigoplus_{k=1}^\ell\lie
g^{\sigma}_{\boa_k^m,N},$$ and hence $\otimes_{k=1}^\ell
W(\bpi_k^\sigma)$ is cyclic for $ \lie g^{\sigma}_{\boa^m,N}$, i.e
$$\bu(\lie g^{\sigma}_{\boa^m,N})(\otimes_{k=1}^\ell
w_{\bpi^\sigma_k})=\bigotimes_{k=1}^\ell W(\bpi^\sigma_k).$$ This
proves that  $\eta$ is  a surjective  map  of $\lie
g^\sigma_{\boa^m, N}$--modules and the proof of the proposition is
complete.

\end{proof}
\subsection{}

We now prove (iii) of Theorem \ref{twisted}. Recall that in
Section 2.1, we have identified elements of $P_\sigma^+$ with
elements of $P^+$ and hence for each $a\in\bc^\times$ and
$\lambda\in P^+_\sigma$ we have elements $\bpi_{\lambda,a}\in\cal
P^+$ and $\bpi_{\lambda,a}^\sigma\in\cal P^+_\sigma$. Moreover,
$\bpi_{\lambda,a}\in\boi(\bpi_{\lambda,a}^\sigma)$.
\begin{proof}
  Choose
$b_\epsilon\in\bc^\times$, $0\le \epsilon\le m-1$  such that
$$b_r\ne b_s,\ \ b_r^m \neq b_s^m,\ \  r\ne s.$$   Using Lemma \ref{standlemtwisted}, Theorem
\ref{twisted}(ii) and  Theorem \ref{twisted}(i) in that order
gives,
 $$
\bigotimes_{\epsilon=0}^{m-1} W(\bpi^\sigma_{\lambda_{\epsilon},
\zeta^\ep a}) \cong_{\lie g_0} \bigotimes_{\epsilon=0}^{m-1}
W(\bpi^\sigma_{\lambda_{\epsilon},
b_{\epsilon}})\cong_{L^{\sigma}(\lie g)} W(\prod
\bpi^\sigma_{\lambda_{\epsilon}, b_{\epsilon}}).$$ Since
$\lambda_\epsilon\in\cal P^+_\sigma$, we have $\prod
\bpi_{\sigma^{\epsilon}(\lambda_{\epsilon}), \zeta^{\epsilon}
b_{\epsilon}}  \in \boi(\prod \bpi^\sigma_{\lambda_{\epsilon},
b_{\epsilon}})$ and so by Theorem \ref{twisted}(i) we get,
$$
W(\prod
\bpi^\sigma_{\lambda_{\epsilon},
b_{\epsilon}})\cong_{L^{\sigma}(\lie g)} W(\prod
\bpi_{\sigma^{\epsilon}(\lambda_{\epsilon}), \zeta^{\epsilon}
b_{\epsilon}}).
$$
Theorem \ref{affine} gives, $$W(\prod
\bpi_{\sigma^{\epsilon}(\lambda_{\epsilon}), \zeta^{\epsilon}
b_{\epsilon}})\cong_{\lie g}W(\prod
\bpi_{\sigma^{\epsilon}(\lambda_{\epsilon}), 1})\cong_{\lie g}
W(\bpi_{\sum_{\epsilon=0}^{m-1}\sigma^{\epsilon}(\lambda_\epsilon),a}).$$
And  since
$\bpi_{\sum_{\epsilon=0}^{m-1}\sigma^{\epsilon}(\lambda_\epsilon),a}
\in \boi(\prod_{\ep = 0}^{m-1}\bpi^\sigma_{\lambda_\ep, \zeta^\ep a})$, we get
$$W(\prod
\bpi_{\sigma^{\epsilon}(\lambda_{\epsilon}), \zeta^{\epsilon}
b_{\epsilon}})\cong_{\lie
g_0}W(\bpi_{\lambda,a})\cong_{L^\sigma(\lie g)}
W(\prod\limits_{\epsilon = 0}^{m-1}
\bpi^\sigma_{\lambda_{\epsilon}, \zeta^\ep a}), $$ which completes
the proof.

\end{proof}
\subsection{} We now  prove Theorem \ref{twisted} (iv).
By Theorem \ref{twisted}(i), we have
$$W^\sigma(\bpi^{\sigma}_{\lambda,a}) \cong_{L^\sigma(\lie g)} W(\bpi_{\lambda,a}).$$   Theorem \ref{affine} gives if  $\lie g$ not of type $A_{2n}$
that $$ W(\bpi_{\lambda,a}) \cong_{\lie g} \bigotimes_{i=1}^{n}
W(\bpi_{\omega_i,1})^{\otimes m_i}$$ and for $\lie g$ of type
gives $A_{2n}$
$$W(\bpi_{\lambda,a})\cong_{\lie g} W(\bpi_{2\omega_n,1})^{\otimes \frac{m_n}{2}} \otimes \bigotimes_{i =1}^{n-1}
W(\bpi_{\omega_i,1})^{\otimes m_i},$$ which completes the proof.

\subsection{} We now prove Theorem \ref{twisted}(v). This part of
the proof is very similar to the one given in \cite{CPweyl} in the
untwisted case and we shall only give a sketch of the proof. Thus,
let $V$ be an $L^\sigma(\lie g)$--module, assume that $V$ is
finite--dimensional and that it is generated by an element $v\in
V$ such that $$L^\sigma(\lie n^+)v=0,\ \ \bu(L^\sigma(\lie
h))v=\bc v.$$ Let $\lambda\in P_\sigma^+$ be such that
$hv=\lambda(h)v$ for all $h\in\lie h_0$. Since $V$ is
finite--dimensional it follows from the representation theory of
the subalgebras $\{x^\pm_{i,0}, h_{i,0}\}$, $i\in I_0$  that
$\lambda\in P^+_\sigma$ and also that
\begin{equation}\label{xim} (x_{i,0}^-)^{s}=0,\ \, i\in I_0,\  \ s\in\bz_+,\ \  s\ge \lambda(h_i)+1\end{equation} Moreover if
$\lie g$ is of type $A_{2n}$, we find by working with the
subalgebra $\left\{ \frac{1}{2}h_{n,0}, y_{n,1}^\pm\otimes t^{\mp
1}\right\}$ that
\begin{equation}\label{a2n}(y^-_{n,1}\otimes t)^{s}v = 0,\ \ s\in\bz_+,\ \ s\ge \frac{1}{2}\lambda(h_{n,0}) +1. \end{equation}
Applying $(x_{i,0}^+\otimes t)^{s}$ to both sides of \eqref{xim},
($i\ne n$ if $\lie g$ of type $A_{2n}$) we find by using Lemma 3.3
(i), (ii),  that
$$\left(\bop^+_{i,\sigma}(u)\right)_{_{s}} = 0, \ \ \ s > \lambda(h_{i,0}),$$
while if $\lie g$ is of type $A_{2n}$, we  apply
$(x^+_{n,0})^{2s}$ to both sides of \eqref{a2n} and using Lemma
3.3(iii), we find
$$\left(\bop^+_{n,\sigma}(u)\right)_{_{k}} = 0, \ \ \ k >  \lambda(\frac{1}{2}h_{n,0}).$$
Set $$\pi_i^\sigma(u)=\sum_{k=0}^\infty(\bop_i^\sigma(u))_ku^k,$$
and let $\bpi^\sigma=(\pi_i^\sigma)_{i\in I_0}$. The preceding
arguments show that $\bpi^\sigma$ is an $I_0$--tuple of
polynomials. We claim  that
\begin{equation}\label{left} \lambda=\lambda_\bpi,\ \ \ \ \bop_{i,\sigma}^-(u)
v=(\pi_i^\sigma(u))^- v,\end{equation} which now shows that $V$ is
a quotient of $W(\bpi^\sigma)$. To prove that
$\lambda=\lambda_\bpi$ is equivalent to proving that
\begin{equation}\label{notan} \left(\bop^+_{i,\sigma}\right)_{\lambda(h_i)}v\ne 0,\end{equation}for all
$i\in I$, if  if $\lie g$ is not of type $A_{2n}$ and for all
$i\ne n$ if $\lie g$ is of type $A_{2n}$ and if $\lie g$ is of
type $A_{2n}$
\begin{equation}\label{an} \left(\bop^+_{n,\sigma}(u)\right)_{\frac{1}{2}\lambda(h_{n,0})}.v\ne
0.\end{equation} It is now easy to see (keeping in mind that
$(\bop_{i\sigma}(u))_{_{0}}=1$) that the following Lemma implies
\eqref{left}.

\begin{lem} Let $V$ be a finite--dimensional $L^\sigma(\lie
g)$--module and let $v\in V_\lambda)$ be such that $L^\sigma(\lie
n^+)v=0$.  For all $i \in I_0$ ($i \neq n$ for $\lie g$ of type
$A_{2n}$), we have\\
\begin{align*}
&(\bop^+_{i, \sigma}(u))_{\lambda(h_{i,0})}(\bop^-_{i,
\sigma}(u))_{k}.v = (\bop^+_{i,
\sigma}(u))_{\lambda(h_{i,0})-k}.v, \ \ \ 0 \leq k \leq
\lambda(h_{i,0}),
\end{align*}
and for $\lie g$ of type $A_{2n}$, we have
\begin{align*}
&(\bop^+_{n, \sigma}(u))_{\frac{1}{2}\lambda(h_{n,0})}(\bop^-_{n, \sigma}(u))_{k}.v =
 (\bop^+_{n, \sigma}(u))_{\frac{1}{2}\lambda(h_{n,0})-k}.v, \ \ \ 0 \leq k \leq
 \frac{1}{2}\lambda(h_{n,0}).
\end{align*}
\end{lem}
\begin{pf} The proof of the  first statement is given in \cite[Proposition 1.1]{CPweyl}
and the key ingredient in that proof  is Lemma 3.4 (i). The proof
when $i=n$ and $\lie g$ of type $A_{2n}$  is entirely similar and
one uses Lemma 3.4 (iii)(a) with $r =
\frac{1}{2}\lambda(h_{n,0})+1$.

\end{pf}


\begin{thebibliography}{999999}
\bibitem{Akasaka} T.~Akasaka {\em An integral PBW basis of the quantum affine
algebra of type $A\sp {(2)}\sb 2$} Publ. Res. Inst. Math. Sci. 38,
no. 4 (2002), 803–894
\bibitem{BN} J.~Beck and H.~Nakajima {\em Crystal bases and two-sided cells of
quantum affine algebras}, Duke Math. J.  123, no. 2 (2004),
335–402


\bibitem{Cferm} V.~Chari {\em On the fermionic formula and the Kirillov-Reshetikhin conjecture},  Internat. Math. Res. Notices  2001,  no. 12, 629--654.
\bibitem{CL} V.~Chari and  S.~Loktev Weyl {\em Demazure and fusion modules for the current algebra of $\lie {sl}_{r+1}$},  Adv. Math.  207  (2006),  no. 2, 928--960.
\bibitem{CPtw} V.~Chari and  A.~Pressley {\em Twisted quantum affine
algebras}	Commun. Math. Phys. 196, no. 2, pp. 461-476 (1998).
\bibitem{CPweyl} V.~Chari and  A.~Pressley {\em Weyl modules for classical and quantum affine algebras}, Represent. Theory  5  (2001), 191--223.
\bibitem{FL} B.~Feigin and S.~Loktev {\em On Generalized Kostka Polynomials and
the Quantum Verlinde Rule}, Differential topology,
infinite-dimensional Lie algebras, and applications, Amer. Math.
Soc. Transl. Ser. 2 Vol. \textbf{194} (1999), p.61-79.
\bibitem{FKL}
B.~Feigin, A.~N.~Kirillov and S.~Loktev {\em Combinatorics and
geometry of higher level Weyl modules}, math.QA/0503315.
\bibitem{FV} T.Fisher-Vasta {\em Presentations of Z-Forms for the Universal Enveloping Algebras of Affine Lie Algebras
} University of California Riverside Ph.D. Dissertation (1999).
\bibitem{FoL} G.~Fourier and  P.~Littelmann {\em  Weyl modules, Demazure modules, KR-modules, crystals, fusion products and limit constructions} (2005), to appear in Advances in Mathematics
\bibitem{Ga} H. Garland {\em The Arithmetic Theory of Loop Algebras}, J. Algebra 53 (1978), 480 - 551.
\bibitem{K} M.~Kashiwara, {\em Crystal bases of modified quantized enveloping algebra},
Duke Math. J. 73 (1994), 383Ð413.
\bibitem{Mitz} D. Mitzman {\em Integral Bases for Affine Lie Algebras and Their Universal Enveloping Algebras}, Contemporary Mathematics 40 (1983).
\bibitem{Nak} H. Nakajima, {\em Quiver varieties and finite-dimensional
representations of quantum affine algebras}, J. Amer. Math. Soc. 14
(2001), no. 1, 145--238, math.RT/0308156.
\end{thebibliography}
\end{document}